\documentclass [11pt]{amsart}  
\usepackage{a4}                
\usepackage{paralist}           
\usepackage{graphicx}
\usepackage{amssymb}                        
\usepackage{float}
\usepackage{amsmath}
\usepackage{psfrag}
\usepackage{MnSymbol}
\usepackage[all]{xy}

\usepackage{amsmath,amsthm}

\theoremstyle{plain} 
\newtheorem{thm}{Theorem}[section]

\newtheorem{lem}[thm]{Lemma}
\newtheorem*{quesb}{Question}
\newtheorem{prop}[thm]{Proposition}
\newtheorem{cor}[thm]{Corollary}

\theoremstyle{definition}
\newtheorem{defn}[thm]{Definition}

\theoremstyle{remark}

\newtheoremstyle{TheoremNum}
        {\topsep}{\topsep}              
        {\itshape}                      
        {}                              
        {\bfseries}                     
        {.}                             
        { }                             
        {\thmname{#1}\thmnote{ \bfseries #3}}
    \theoremstyle{TheoremNum}
    \newtheorem{thmn}{Theorem}
\newtheorem{corn}{Corollary}

\DeclareMathOperator{\tor}{Tor}
\DeclareMathOperator{\rank}{rank}

\DeclareMathOperator{\torlen}{TorLen}

\DeclareMathOperator{\tgen}{fg}
\DeclareMathOperator{\ord}{o}

\title[Torsion length]{A note on torsion length}
\author{Maurice Chiodo, Rishi Vyas}

\begin{document}

\let\thefootnote\relax\footnotetext{2010 \textit{AMS Classification:} 20F05, 20E06, 20F14.}
\let\thefootnote\relax
\let\thefootnote\relax\footnotetext{The first author was partially supported by the Italian FIRB ``Futuro in Ricerca'' project RBFR10DGUA\_002 at the University of Milan.}
\let\thefootnote\relax\footnotetext{The second author did a part of this work while writing his thesis at the University of Cambridge. He thanks Wolfson College, Cambridge and the Cambridge Commonwealth Trust for their support.}

\begin{abstract}
We construct a $2$-generator recursively presented group with infinite torsion length. We also explore the construction in the context of solvable and word-hyperbolic groups.
\end{abstract}

\maketitle

What should the `torsion' subgroup of an arbitrary group be? The set of torsion elements does not work: it is not necessarily a subgroup. Attempting to consider the subgroup generated by the set of torsion elements as the `torsion' subgroup of a group is also unsatisfactory; the quotient of a group by this subgroup need not be torsion-free, as shown in proposition \ref{first eg}. We can, however, iterate this procedure: letting $\tor_{1}(G)$ be the subgroup generated by the torsion elements of a group $G$, we inductively define $\tor_{n+1}(G)/\tor_{n}(G)=\tor_{1}(G/\tor_{n}(G))$, and form the union $\tor_{\omega}(G)=\bigcup_{n\in \mathbb{N}}\tor_{n}(G)$. The subgroup $\tor_{\omega}(G)$ is a viable candidate for the `torsion' subgroup of a group.  The structure of  $\tor_{\omega}(G)$ as a countable union of subgroups allows us to attach an invariant to any group $G$, which we call the \emph{torsion length} of $G$ (definition \ref{tordefn}) and denote by $\torlen(G)$: this is the minimum ordinal $n$ such that $\tor_{n}(G)=\tor_{\omega}(G)$. All of this is described in greater detail in \textsection \ref{torsec}.

In \textsection \ref{embedsec}, we study embeddings. There is a well-known, uniform process for embedding a countable group into a $2$-generator group. We describe this process (lemma \ref{2 gen tor}) and verify that it does not change torsion length (theorem \ref{3 gen torlen}).

In \textsection \ref{groupswitharblen}, we begin by constructing finitely presented groups with arbitrary finite torsion length. More precisely, we prove the following result (writing $\overline{P}$ to denote the group presented by a presentation $P$):

\vspace{3mm}

\begin{thmn}[\ref{fp tor len}]
There is a family of finite presentations $\{P_{n}\}_{n \in \mathbb{N}}$ such that:
\\$1$. $\overline{P}_{n+1} / \tor_{1}(\overline{P}_{n+1}) \cong \overline{P}_{n}$,
\\$2$. $\torlen(\overline{P}_{n})=n$.
\end{thmn}
\vspace{3mm}

We apply a result of  Kharlampovich and Myasnikov (\cite[Corollary 2]{KharMia}) to show that the groups presented by the examples constructed in theorem \ref{fp tor len} are word hyperbolic (proposition \ref{conjsepto}). 

The group-theoretic construction of the $\overline{P}_{n}$ also appears in Cirio \emph{et.~al.} \cite[Example 5.16]{Cirio et al}. Moreover, Leary and Nucinkis \cite[\S5 Corollary 7]{LearyNuc} give an alternative group-theoretic construction, which we describe in theorem \ref{LN construction}. 

We then use these constructions along with results from \textsection \ref{torsec} to prove our main result:

\vspace{3mm}
\begin{thmn}[\ref{fp inf torlen}]
There exists a $2$-generator recursive presentation $Q$ for which $\torlen(\overline{Q}) = \omega$.
\end{thmn}
\vspace{3mm}

Are there finitely presented groups of infinite torsion length? We do not know, but consider this an interesting question for future research.

Every nilpotent group has torsion length at most $1$ (\cite[5.2.7]{Rob}). In \textsection \ref{solsec}, we show that this is not necessarily the case for polycyclic groups: 
\vspace{3mm}
\begin{corn}[\ref{sol eg cor}]
There exists a finitely presented polycyclic group of torsion length $2$.
\end{corn}

\vspace{3mm}

Unfortunately, we have been unable to construct finitely generated solvable groups of torsion length greater than two. One can also ask if there exist finitely generated solvable groups of infinite torsion length. 

\subsection*{Acknowledgements} We would like to thank Claudia Pinzari for her interest in our work, Jack Button and Andrew Glass for their comments and suggestions, and Ian Leary for his thoughtful conversations.

\section{Preliminaries}  \label{prelim}

\subsection{Notation}

If $P$ is a group presentation, we denote by $\overline{P}$ the group presented by $P$. A presentation $P=\langle X|R\rangle$ is said to be a \emph{recursive presentation} if $X$ is a finite set and $R$ is a recursive enumeration of relations; $P$ is said to be a \emph{countably generated recursive presentation} if instead $X$ is a recursive enumeration of generators. 
A group $G$ is said to be \emph{finitely} (respectively,   \emph{recursively}) \emph{presentable} if $G\cong \overline{P}$ for some finite (respectively,   recursive) presentation $P$. 
If $P,Q$ are group presentations then we denote their free product presentation by $P*Q$: this is given by taking the disjoint union of their generators and relations. If $g_{1}, \ldots, g_{n}$ are elements of a group $G$, we write $\langle   g_{1}, \ldots, g_{n} \rangle$ for the subgroup in $G$ generated by these elements and $\llangle g_{1}, \ldots, g_{n} \rrangle^{G}$ for the normal closure of these elements in $G$. Let $\omega$ denote the smallest infinite ordinal. Let $\ord(g)$ denote the order of a group element $g$; recall that $g\in G$ is \emph{torsion} if $1\leq \ord(g) <\omega$. Write $\tor(G):=\{g \in G\ |\ g \textnormal{ is torsion}\}$.  Let $|X|$ denote the cardinality of a set $X$. If $X$ is a set, let $X^{-1}$ be a set of the same cardinality as and disjoint from $X$ along with a fixed bijection ${*}^{-1}: X \to X^{-1}$. Write $X^{*}$ for the set of finite words on $X \cup X^{-1}$.

\subsection{Preliminary facts in group theory}

We now collect a few lemmas that we will need later in this paper. They all must be well known, but we have been unable to find suitable references.

\begin{lem} \label{freepfrees}
Let $G_{1}$ and $G_{2}$ be non-trivial groups, and suppose $|G_{1}|>2$. Then, $G_{1}*G_{2}$ contains a non-abelian free subgroup.
\begin{proof}
Since $|G_{1}|>2$, there exist $y$ and $z$ such that $yz\neq e$. Let $x$ be a non trivial element of $G_{2}$. The reader will easily check that the elements $yxz$ and $xyxzx$ freely generate a free subgroup in $G_{1}*G_{2}$.
\end{proof}
\end{lem}

\begin{lem} \label{c2sol}
$C_{2}*C_{2}$ is polycyclic.
\begin{proof}
Let the two copies of $C_{2}$ be generated by $x$ and $y$ respectively. Then, since $xxyx=yx$ and $yxyy=yx$, it follows that the cyclic subgroup generated by $xy$ is normal. It is also easy to see that $C_{2}*C_{2}/{\langle xy \rangle}\cong C_{2}.$ Thus $C_{2}*C_{2}$ is polycyclic.
\end{proof}
\end{lem}

\begin{defn}[\cite{KharMia}]
A subgroup $H$ of a group $G$ is said to be \textit{conjugate separated} if for any $x\in G \setminus H$ we have that $H\cap xHx^{-1}$ is finite.
\end{defn}

\begin{lem} \label{conjsep}
Let $A$ and $B$ be groups, and suppose $e\neq a\in A$ and $e\neq b\in B$, with either $\ord(a)\neq 2$ or $\ord(b)\neq 2$. Then for any $x\in A*B \setminus \langle ab \rangle$, $\langle ab \rangle \cap x \langle ab \rangle x^{-1}=\{e\}$. Hence $\langle ab \rangle$ is conjugate separated in $A*B$. Moreover,  if $\ord(a) = \ord(b)=2$, then the subgroup $\langle ab \rangle$ is not conjugate separated.
\begin{proof}
Without loss of generality we may take $\ord(b)\neq 2$. Suppose that  $\langle ab \rangle$ is not conjugate separated in $A*B$. Then, there exists an $x\in A*B\setminus \langle ab \rangle$, and $i, j\in \mathbb{Z} \setminus\{0\}$ such that $x(ab)^{i}x^{-1}=(ab)^{j}$. We can assume that the underlying word of $x$ is reduced. We will induct on the length of $x$ as a reduced word. 
\\Let us first assume that $i > 0$. It follows that there must exist $x'\in A*B$ such that either $x=x'a^{-1}$, or $x=x'b$. This is true because any other eventuality would lead to $x(ab)^{i}x^{-1}$ having a reduced underlying word which begins and ends with a letter from the same group, and an element with such a word clearly cannot belong to $\langle ab \rangle$. Since $\ord(b)\neq 2$, $\langle ab \rangle \cap \langle ba \rangle = \{e\}$. It follows that $x'\neq e$, as $a^{-1}\langle ab \rangle a = b\langle ab \rangle b^{-1} = \langle ba \rangle$. We therefore have that $x'(ba)^{i}{x'}^{-1}=(ab)^{j}$. Assume that $x=x'a^{-1}$, then by reasoning as we did the previous paragraph, we see that $x'=x''b^{-1}$ (the other option, where $x'=x''a$, cannot happen as that would them imply that $x$ was not reduced.) It follows that $x=x''b^{-1}a^{-1}$. Assuming that $x=x'b$, a similar line of reasoning allows us to reach the conclusion that $x=x''ab$, for some element $x''$. In either case, we have that  $x''\langle ab \rangle^{i}{x''}^{-1}=(ab)^{j}$. Applying our induction hypothesis, we see that $x''\in \langle ab \rangle$. Thus $x\in \langle ab \rangle$. The case where $i < 0$ is analogous. 
\\Finally,  if $\ord(a) = \ord(b)=2$, then $aaba=babb=ba=(ab)^{-1}$.
\end{proof}
\end{lem}

\section{Torsion} \label{torsec}

\subsection{Introduction to torsion length} 

\begin{defn} \label{defoftorsion}
Given a group $G$, we inductively define $\tor_{n}(G)$ as follows:
\[
\tor_{0}(G):=\{e\},
\]
\[
\tor_{n+1}(G):=\llangle\ \{g \in G\ |\ g\tor_{n}(G) \in \tor \big( G/\tor_{n}(G)\big) \}\ \rrangle ^{G},
\]
\[
\tor_{\omega}(G):=\bigcup_{n \in \mathbb{N}}\tor_{n}(G).
\]
\end{defn}

Definition \ref{defoftorsion} first appeared in \cite{Chiodo2}. The following lemma also appears as \cite[Proposition 4.9]{Chiodo2}.

\begin{lem} \label{leftadj}
$G/\tor_{\omega}(G)$ is torsion-free. Moreover, if $f:G\rightarrow H$ is a group homomorphism from $G$ to a torsion-free group $H$, then $\tor_{\omega}(G)\leq \ker(f)$.
\begin{proof}
If $x^{n}\in \tor_{\omega}(G)$ for some $n >0 $, then there exists $m\in \mathbb{N}$ such that $x^{n}\in \tor_{m}(G)$. It follows that $x \in \tor_{m+1}(G)$, and thus that $x\in \tor_{\omega}(G)$. Thus $G/\tor_{\omega}(G)$ is torsion-free. 
\\If $f:G\rightarrow H$ is a group homomorphism from $G$ to a torsion-free group $H$, it follows that $\tor(G)\leq \ker(f)$, and thus that $\tor_{1}(G)\leq \ker(f)$. Then, $f$ factors through $G/\tor_{1}(G)$. By induction, we see that $\tor_{n}(G)\leq \ker(f)$ for all $n$, and thus that $\tor_{\omega}(G)\leq \ker(f)$.
\end{proof}
\end{lem}

The following lemma records some facts that will be used later. We leave the proof as an exercise.

\begin{lem}\label{normal} \label{simplify} \label{pre induct} \label{tor under sg} \label{interchange}
Let $G$ be a group, with $H\leqslant G$. Let $i, j \in \mathbb{N}$.
\\a.$)$ $\tor_{i+1}(G)=\langle\ \{g \in G\ |\ g\tor_{i}(G) \in \tor \big( G/\tor_{i}(G)\big) \}\ \rangle.$ 
\\b.$)$ $\tor_{i+1}(G)=\langle\ \{g \in G\ |\ g^n \in \tor_{i}(G)\textnormal{ for some }n>0  \}\ \rangle.$ 
\\c.$)$ $\tor_{i+1}(G)/\tor_{i}(G)=\tor_{1} \big( G/\tor_{i}(G) \big)$ as subgroups of $G/\tor_{i}(G)$. 
\\d.$)$ $\big(G/\tor_{i}(G)\big)/\tor_{j} \big( G/\tor_{i}(G) \big) \cong G/\tor_{i+j}(G)$ via the obvious quotient map.
\\e.$)$ $\tor_{i}(H)\leqslant \tor_{i}(G)$.
\end{lem}

The following is a standard result from combinatorial group theory. 

\begin{lem}\label{enum Tor i}
Let $P= \langle X|R \rangle$ be a recursive presentation. Then the elements of $X^{*}$ which represent elements in the subgroup $\tor_{i}(\overline {P})$ are recursively enumerable, uniformly in $P$ and in each $i \in \mathbb{N}$. Hence the words representing elements in $\tor_{\omega}(\overline {P})$ are also recursively enumerable.
\end{lem}

We make the following definition; the same notion appears in Cirio \emph{et.~al.} \cite{Cirio et al} as the \emph{Torsion Degree} (\cite[Definition 5.5]{Cirio et al}) of a group.

\begin{defn} \label{tordefn}
We define the \emph{Torsion Length} of $G$, $\torlen(G)$, by the smallest ordinal $n$ such that $\tor_{n}(G)=\tor_{\omega}(G)$. 
\end{defn}

The following lemma summarizes some basic properties of this invariant. Again, we omit the proof.

\begin{lem}\label{sum lengths}\mbox{}
\\a.$)$ If $G$ is a non-trivial torsion group $($i.e. $\tor(G)=G)$, then $\torlen(G)=1$.
\\b.$)$ If  $n\leq \torlen(G)$, then  $\torlen \big(G/\tor_{n}(G)\big)=\torlen(G)-n$, with the convention $\omega -n = \omega$ and $\omega -\omega =0$.
\\c.$)$ $\torlen(G)$ is the smallest ordinal $n$ for which $G/\tor_{n}(G)$ is torsion-free.
\end{lem}

The following result describes the behavior of torsion in amalgamated products.

\begin{thm}[{\cite[Theorem 11.69]{Rot}}]\mbox{}\label{tor thm}
Let $g\in \tor(G)$. Then:
\\$1$. If $G=K_{1} *_{H}K_{2}$ is an amalgamated product, then $g$ is conjugate to an element of $K_{1}$ or $K_{2}$.
\\$2$. If $G=K*_{H}  $ is an HNN extension, then $g$ is conjugate to an element in the base group $K$.
\end{thm}

If $\{A_{i}\}_{i\in I}$ is a family of groups, write $*_{i\in I}A_{i}$ for the free product of all the $A_{i}$. 

\begin{cor}\label{tor cor}
Suppose $g \in \tor(*_{i\in I}A_{i})$, where $I$ is any index set. Then $g$ is conjugate to an element in one of the $A_{i}$'s.
\end{cor}

\begin{cor}\label{amalgams}
Let $A,B$ be groups, and $H$ a group which embeds into both $A$ and $B$. Then $\tor_{1}(A*_{H}B)=\llangle \tor(A) \cup \tor(B) \rrangle^{A*_{H}B}$ (where $\tor(A)$ and $\tor(B)$ are viewed as subsets of $A,B$ respectively, and hence as subsets of $A*_{H}B$ under the natural embeddings).
\end{cor}

We can extend the above proposition, if we restrict ourselves to free products \emph{without} amalgamation. 

\begin{prop}\label{tor through products}
Let $\{A_{i}\}_{i\in I}$ be a family of groups.  Then, for all ordinals $j\leq \omega$, $\tor_{j}(*_{i\in I}A_{i})= \llangle \cup_{i\in I} \tor_{j}(A_{i}) \rrangle^{*_{i\in I}A_{i}}$, and the natural map 
\[
*_{i\in I} (A_{i}/\tor_{j}(A_{i})) \rightarrow (*_{i\in I}A_{i})/\tor_{j}(*_{i\in I}A_{i})
\]
is an isomorphism.

\begin{proof}
Set $F:= *_{i\in I}A_{i}$. The fact that  $\tor_{1}(F)= \llangle \cup_{i\in I} \tor_{1}(A_{i}) \rrangle^{F}$ follows immediately from corollary \ref{tor cor}. Using this, it follows easily that the natural map 
\[
*_{i\in I} (A_{i}/\tor_{1}(A_{i}))\rightarrow (F/\tor_{1}(F))
\]
is an isomorphism. Using lemma \ref{normal}, and induction, we see that $\tor_{j}(F)= \llangle \cup_{i\in I} \tor_{j}(A_{i}) \rrangle^{F}$ for all $j\in \mathbb{N}$. The statement for $\tor_{\omega}$ now follows by taking unions.
The fact that the natural map 
\[
*_{i\in I}( A_{i}/\tor_{j}(A_{i}))\rightarrow (F/\tor_{j}(F))
\]
is an isomorphism is an immediate consequence. 
\end{proof}
\end{prop}

\begin{cor}\label{prod cor}
Let $\{A_{i}\}_{i\in I}$ be a family of groups. Then 
\[
\torlen(*_{i\in I} A_{i})= \sup\{\torlen(A_{i})\}_{i\in I}.
\]
\end{cor}

\subsection{Embeddings} \label{embedsec}

The following result is well-known (see, for example, \cite[Corollary 11.72]{Rot}). 

\begin{lem}\label{2 gen tor}\label{2 gen finite}
There is a uniform procedure that, on input of any countably generated recursive presentation $P= \langle X|R \rangle$, outputs a $2$-generator recursive presentation $\tgen(P)$ such that $\overline{P}$ embeds in $\overline{\tgen(P)}$. Moreover, when $P$ is a finite presentation, $\tgen(P)$ is a finite presentation.
\end{lem}

\begin{proof}
Fix an enumeration $x_{1}, x_{2}, \ldots$ of all letters in $X$. 
Let $P_{1}:=\langle a,b | - \rangle$ be a presentation for the free group $F_{2}$. Consider the following two subgroups of $\overline{P}*\overline{P}_{1}$:
\[
 A:= \langle a, x_{1}b^{-1}ab, x_{2}b^{-2}ab^{2}, \ldots, x_{i}b^{-i}ab^{i}, \ldots \rangle 
\]
and
\[
 B:= \langle  b, a^{-1}ba, \ldots, a^{-i}ba^{i}, \ldots \rangle, 
\]
where in each case $i$ ranges over all values for which $x_{i}\in X$.
\\Note that the sets $\{ a^{-i}ba^{i} \}_{i \in \mathbb{N}}$ and $\{ b^{-i}ab^{i} \}_{i \in \mathbb{N}}$ 
freely generate copies of $F_{\omega}$ in $\overline{P}_{1}$. 
Thus, by the normal form theorem for free products, $\{ x_{i} b^{-i}ab^{i} \}_{i \in \mathbb{N}}$ 
freely generates a copy of $F_{\omega}$ in $\overline{P}*\overline{P}_{1}$. 
Thus $A$ and $B$ are isomorphic, and such an isomorphism can be given by the extension $\overline{\phi}$ of the set map $\phi(a):=b$, $\phi(x_{i} b^{-i}ab^{i}):=a^{-i}ba^{i}$ for all $i$ where $x_{i} \in X$. 
We now form the HNN extension $\overline{P}*_{\overline{\phi}}$ of $\overline{P}$, conjugating $A$ to $B$. 
This can be realised via the following presentation:
\[
Q:=\langle X, a, b, t | R,\ t^{-1}at=b,\ t^{-1}x_{i}b^{-i}ab^{i}t=a^{-i}ba^{i}\ \forall i \textnormal{ with } x_{i} \in X \rangle.
\]
It is not hard to see that $\overline{Q}$ is generated by $a$ and $t$. Removing $X$ and $b$ from the generating set of $Q$, and making the relevant substitutions in the relating set of $Q$ gives us our desired $2$-generator recursive presentation, which we denote by $\tgen(P)$; by construction it is then clear that $\overline{P}$ embeds in $\overline{\tgen(P)}$. Finally, if $P$ is a finite presentation, then $Q$ will be a finite presentation, and hence so will $\tgen(P)$.
\end{proof}

\begin{lem} \label{lemlem}
Let $P=\langle X|R \rangle$ be a countably generated recursive presentation. Let $S$ be a recursive enumeration of a subset of $X^{*}.$ Then $\tgen(\langle X|R \cup S \rangle)$ is the presentation $\tgen(\langle X|R \rangle)$ with $S$ adjoined to its relating set.
\end{lem}

\begin{proof}
The construction of $\tgen(\langle X|R \rangle)$ is completely uniform in the relating set $R$. Thus we can add relations either before or after the amalgamation step, and it does not change the final presentation.
\end{proof}

Of course, in the above result we need to be careful about the notion of the union of two recursive enumerations of elements, as a recursive enumeration.

\begin{cor}
Let $P=\langle X|R \rangle$ be a countably generated recursive presentation. Let $S$ be a recursive enumeration of a subset of $X^{*}.$ Then
\[
{\overline{\tgen(\langle X|R \cup S \rangle)} \cong \overline{\tgen(P)}/\llangle S \rrangle^{\overline{\tgen(P)}}}. 
\]
\end{cor}

\begin{cor}\label{preserved}
Let $P=\langle X|R \rangle$ be a countably generated recursive presentation. Take an enumeration $T_{i}$ of all elements of $X^{*}$ representing elements of $\tor_{i}(\overline {P})$ (lemma $\ref{enum Tor i}$).
Then
\[
 \overline{\tgen(\langle X|R \cup T_{i} \rangle)} \cong  \overline{\tgen(P)}/\tor_{i}(\overline{\tgen(P)}).
\]
\end{cor}

\begin{lem}\label{pres tor free}
Let $P$ be a countably generated recursive presentation. Then $\tor_{1}(\overline{\tgen(P)})= \llangle \tor({\overline{P}}) \rrangle^{\overline{\tgen(P)}}$.
\end{lem}

\begin{proof}
By theorem \ref{tor thm}, every torsion element in $\overline{\tgen(P)}$ is conjugate to a torsion element in $\overline{P}$. 
\end{proof}

\begin{thm}\label{3 gen torlen}
Let $P$ be a countably generated recursive presentation (respectively,  finite presentation). Then we can construct a $2$-generator recursive presentation (respectively,  finite presentation) $\tgen(P)$ as given in lemma $\ref{2 gen tor}$, uniformly in $P$, such that $\overline{P}$ embeds in $\overline{\tgen(P)}$, and $\torlen(\overline{\tgen(P)})=\torlen(\overline{P})$.
\end{thm}

\begin{proof}

The first part of the theorem is proved in lemma \ref{2 gen tor}. All that remains to be shown is that $\torlen(\overline{\tgen(P)})=\torlen(\overline{P})$. 
By corollary \ref{preserved}, for any $i \leq \omega$, we have that $\overline{\tgen(\langle X|R \cup T_{i} \rangle)} \cong \overline{\tgen(P)}/\tor_{i}(\overline{\tgen(P)})$ (where $T_{i}$ is an enumeration of all words in $X^{*}$ representing elements in $\tor_{i}(\overline {P})$, via lemma \ref{enum Tor i}). By lemma \ref{pres tor free}, $\overline{\tgen(\langle X|R \cup T_{i} \rangle)}$ is torsion-free if and only if $\overline{\langle X|R \cup T_{i} \rangle}$ is. Since $\overline{P}/\tor_{i}(\overline{P}) \cong \overline{\langle X|R \cup T_{i} \rangle}$, we get that $\overline{\tgen(P)}/\tor_{i}(\overline{\tgen(P)})$ is torsion-free if and only if $\overline{P}/\tor_{i}(\overline{P})$ is. The result now follows, using lemma \ref{sum lengths}.
\end{proof}

\section{Constructions} \label{groupswitharblen}

\subsection{Groups of arbitrary finite torsion length}

\begin{prop}[{\cite[Proposition 4.10]{Chiodo2}}]\label{first eg}
Given any $j,k,l>1$, we can define the finite presentation
\[ P_{j,k,l}:=\langle x,y,z | x^j=e,\ y^k=e,\ xy=z^l \rangle.\]
Then $\overline{P}_{j,k,l}/\llangle\tor(\overline{P}_{j,k,l})\rrangle^{\overline{P}_{j,k,l}}\cong C_{l}$ and $\torlen(\overline{P}_{j,k,l})=2$.
\end{prop}

\begin{proof}
As the value of the subscripts on $P_{j,k,l}$ is irrelevant for this argument, we suppress them. It is clear from the presentation $P$ that $\overline{P} \cong (C_{j}*C_{k})*_{\langle xy\rangle =\langle z^l\rangle}\mathbb{Z}$; the amalgamated product of $C_{j}*C_{k}$ and $\mathbb{Z}$ over infinite cyclic subgroups. By corollary \ref{amalgams}, $\tor_{1}(\overline{P})=\llangle \tor(C_{j}*C_{k}) \rrangle^{\overline{P}}$. Moreover, $x,y \in \tor(\overline{P})$, and so $\tor_{1}(\overline{P})=\llangle C_{j}*C_{k} \rrangle^{\overline{P}}$. It follows that $\overline{P}/\tor_{1}(\overline{P})$ has presentation $Q:=\langle x,y,z |x^j=e, \  y^k=e, \  xy=z^l, \  x=e, \  y=e \rangle$, thus $\overline{Q} \cong C_{l}$ and $\torlen(\overline{P})=2$. 
\end{proof}

\begin{defn}
Let $\{0,1\}^{n}$ denote the set of binary strings of length precisely $n$, where we define $\{0,1\}^{0}:=\{\emptyset\}$. If $\eta \in \{0,1\}^{n}$, then we write $\eta 0$ (respectively,   $\eta 1$) for the binary string of length $n+1$, given by appending $0$ (respectively,   $1$) to the rightmost end of $\eta$. Moreover, if $\eta \in \{0,1\}^{n}$, then we write $\eta'$ for the binary string of length $n-1$ given by removing the rightmost digit from $\eta$.
\end{defn}

We thank Claudia Pinzari; her questions led us to the following generalisation of proposition \ref{first eg}.

\begin{thm}\label{fp tor len}
There is a family of finite presentations $\{ P_{n}\}_{n \in \mathbb{N}}$ of groups satisfying $\torlen(\overline{P}_{n})=n$ and $\overline{P}_{n} /  \tor_{1}(\overline{P}_{n}) \cong \overline{P}_{n-1}$. Explicitly, these are:
\[
P_{n}:=\big\langle \ x_{\eta}\ \forall \eta \in \bigcup_{i=0}^{n-1} \{0,1\}^{i}\ \big|\ x_{\eta}^{3}=e\ \forall \eta \in \{0,1\}^{n-1},\ x_{\eta 0}x_{\eta 1}=x_{\eta}^{3}\ \forall \eta \in \bigcup_{i=0}^{n-2} \{0,1\}^{i}\ \big\rangle,
\]
which have $2^{n}-1$ generators, and $2^{n}-1$ relations.
\end{thm}

\begin{proof}
Note that $P_{0}:=\langle -|- \rangle$ and $P_{1}:=\langle x_{\emptyset} | x_{\emptyset}^{3}=e \rangle$. 
We have two copies of $\overline{P}_{n}$ sitting in $\overline{P}_{n+1}$, identifying $x_{\eta}$ in $\overline{P}_{n}$ with $x_{\eta 0}$ and $x_{\eta 1}$ respectively. Letting $r=x_{0}$, $s=x_{1}$, and $t=x_{\emptyset}$, it can then be seen that $\overline{P}_{n+1}\cong (\overline{P}_{n}*\overline{P}_{n})*_{\langle rs\rangle=\langle t^{3}\rangle} (\mathbb{Z})$. 
\\In proposition \ref{first eg} we showed that $\overline{P}_{2}/\tor_{1}(\overline{P}_{2}) \cong \overline{P}_{1}$. Now suppose $\overline{P}_{j}/\tor_{1}(\overline{P}_{j})$ $\cong \overline{P}_{j-1}$ for all $j \leq n$. Writing $\overline{P}_{n+1}\cong (\overline{P}_{n}*\overline{P}_{n})*_{\langle rs\rangle=\langle t^{3}\rangle} (\mathbb{Z})$, we see (by two applications of corollary \ref{amalgams}) that $\tor_{1}(\overline{P}_{n+1})=\llangle \tor(\overline{P}_{n}*\overline{P}_{n})\rrangle^{\overline{P}_{n+1}} =\llangle \tor(\overline{P}_{n}) \cup \tor(\overline{P}_{n})\rrangle^{\overline{P}_{n+1}}$ (the notation here is unfortunate; $\tor(\overline{P}_{n}) \cup \tor(\overline{P}_{n})$ denotes the union of the torsion elements of the two individual factors of $\overline{P}_{n}*\overline{P}_{n}$). Since $n\geq 1$, $r$ and $s$ remain non-trivial in their respective factors of $\overline{P}_{n}/\tor_{1}(\overline{P}_{n})$ and $\langle rs \rangle$ is still infinite cyclic, so the amalgamation is unaffected. By the inductive hypothesis, $\overline{P}_{n}/\tor_{1}(\overline{P}_{n}) \cong \overline{P}_{n-1}$, so we have
\begin{align*}
\overline{P}_{n+1}/\tor_{1}(\overline{P}_{n+1}) & = (\overline{P}_{n}*\overline{P}_{n})*_{\langle rs\rangle=\langle t^{3}\rangle} (\mathbb{Z})/ \llangle \tor(\overline{P}_{n}) \cup \tor(\overline{P}_{n})\rrangle^{\overline{P}_{n+1}}
\\ & \cong \big( (\overline{P}_{n}/\tor_{1}(\overline{P}_{n}))*(\overline{P}_{n}/\tor_{1}(\overline{P}_{n}))\big)*_{\langle rs\rangle=\langle t^{3}\rangle} (\mathbb{Z})
\\ & \cong (\overline{P}_{n-1}*\overline{P}_{n-1})*_{\langle rs\rangle=\langle t^{3}\rangle} (\mathbb{Z})
\\ & \cong \overline{P}_{n}
\end{align*}
which completes the inductive step.
\\Since lemma \ref{sum lengths} tells us that $\torlen(\overline{P}_{n+1})=\torlen(\overline{P}_{n})+1$, it follows that $\torlen(\overline{P}_{n})=n$. The number of generators and relations is self-evident.
\end{proof}

The recursive definition $\overline{P}_{n+1}:= (\overline{P}_{n}*\overline{P}_{n})*_{\langle rs\rangle=\langle t^{3}\rangle} (\mathbb{Z})$ first appeared (as far as we are aware) in \cite[Example 5.16]{Cirio et al} by Cirio \emph{et.~al.} as a generalisation of \cite[Proposition 4.10]{Chiodo2}. Our work is independent of that in \cite{Cirio et al}, but given how natural the extension is, it is unsurprising that the two constructions are the same.

We now show that $\overline{P}_{n}$ is word-hyperbolic for all $n$; we thank Jack Button for suggesting that they might be, and further suggesting the use of theorem \ref{conj sep thm}. 

\begin{thm}[Kharlampovich-Myasnikov, {\cite[Corollary 2]{KharMia}}]\label{conj sep thm}
Let $G_{1}, G_{2}$ be word-hyperbolic groups, and $A \leq G_{1}$, $B \leq G_{2}$ virtually
cyclic subgroups. Then the group $G_{1}*_{A=B}G_{2}$ is word-hyperbolic if and only if either $A$ is conjugate
separated in $G_{1}$ or $B$ is conjugate separated in $G_{2}$.
\end{thm}

\begin{prop} \label{conjsepto}
The groups $\overline{P}_{n}$ constructed in theorem $\ref{fp tor len}$ are word-hyperbolic, for all $n\in \mathbb{N}$. As a consequence, for every $n\in \mathbb{N}$, there exists a finitely presented word-hyperbolic group of torsion length $n$.
\begin{proof}
We proceed by straightforward induction. Firstly, $P_{1}$ is word-hyperbolic, as it is finite. Now,  $\overline{P}_{n+1}\cong (\overline{P}_{n}*\overline{P}_{n})*_{\langle rs\rangle=\langle t^{3}\rangle} (\mathbb{Z})$, where the notation is as in theorem \ref{fp tor len}. As $\overline{P}_{n}$ has no elements of order $2$ (by theorem \ref{tor thm}), we see by lemma \ref{conjsep} that $\langle rs \rangle$ is  conjugate separated in $\overline{P}_{n}*\overline{P}_{n}$. Since $\langle rs \rangle$ and $\langle t^{3} \rangle$ are both cyclic, it follows by theorem \ref{conj sep thm} that $\overline{P}_{n+1}$ is word-hyperbolic. This completes the induction.
\end{proof}
\end{prop}

We now provide another perspective on these matters, using the following construction of Leary and Nucinkis \cite{LearyNuc}. 

\begin{thm}[{\cite[\S5 Corollary 7]{LearyNuc}}]\label{LN construction}
Let $G$ be a finitely generated group. Then, there is a group $\tilde{G}$ and a surjection $\phi :\tilde{G} \to G$ such that $\ker(\phi)= \tor_{1}(\tilde{G})$. A presentation for $\tilde{G}$ can be formed from a presentation for $G$, with the use of $2$  more generators and at most $2$  more relations. Thus if $G$ is finitely presented,  $\tilde{G}$ can also be made to be finitely presented, and in a uniform algorithmic manner.
\end{thm}

\begin{proof}
Let $\psi: F_{k} \to G$ be a surjection from a free group (of minimal rank) to $G$, with kernel $N$. So $N=\llangle R \rrangle^{F_{k}}$ for some set $R \subset F_{k}$ (which can be taken finite if $G$ is finitely presented). Hence $\langle R \rangle$ is a free group. Let $r:=\rank (\langle R \rangle)$ ($r= \omega$ if $G$ is not finitely presented; if $G$ is finitely presented we can effectively compute $r:= \rank(\langle R \rangle)$ and a free generating set of $\langle R \rangle$ using Stallings foldings, see (\cite{KapMia})). The group $C_{2}*C_{3}:=\overline{\langle x,y|x^2,y^3\rangle}$ contains an embedded copy of $F_{2}$, freely generated by $a:=yxy$ and $b:=xyxyx$ (lemma \ref{freepfrees}). Thus the subgroup of $C_{2}*C_{3}$ generated by $S_{\omega}:=\{b^{-1}ab, \ldots, b^{-n}ab^{n}, \ldots\}$ freely generates an embedded copy of $F_{\omega}$. Let  $S_{l}:=\{b^{-1}ab, \ldots, b^{-l}ab^{l}\}$. Now form the free product with  amalgamation
\[
\tilde{G}:=F_{k}*_{\phi}(C_{2}*C_{3})
\]
where $\phi$ identifies $\langle R \rangle$ and $\langle S_{r} \rangle$. Note that $\tor_{1}(\tilde{G})= \llangle x,y \rrangle^{\tilde{G}}$, so annihilating the torsion of $\tilde{G}$ leaves us exactly with $G$ (annihilating $C_{2}*C_{3}$ means we annihilate $S_{r}$, and hence $R$, and hence the normal closure of $R$, which is $N$).
It is easy to see that $\tilde{G}$ can be presented with $2+k$ generators and $2+r$ relations. 
\end{proof}

\begin{cor}\label{other version}
There is a sequence of finitely presented groups $\{G_{n}\}_{n \in \mathbb{N}}$ such that, for each $n$, $\torlen(G_{n})=n$ and $G_{n}\tor_{1}/(G_{n}) \cong G_{n-1}$. Moreover, each $G_{n}$ has a finite presentation with $2n-1$ generators and at most $2n-1$ relations.
\end{cor}

\begin{proof}
Set $G_{1}:= C_{2}$ with finite presentation $\langle z|z^2 \rangle$, and inductively define $G_{n+1}:= \tilde{G_{n}}$ for each $n>1$. The result then follows from theorem \ref{LN construction} and its proof.
\end{proof}

It is unclear whether these groups are word-hyperbolic.

\begin{quesb}
Is there some finite bound $k$ such that, for each $n \in \mathbb{N}$, there is a finite presentation with at most $k$ generators and $k$ relations of a group with torsion length $n$?
\end{quesb}

\subsection{A $2$-generator group with infinite torsion length} \label{mainres}

\begin{lem}\label{non-hopf inf}
Take the finite presentations $P_{0}, P_{1}, \ldots$ from theorem $\ref{fp tor len}$. Form their free product presentation $P := P_{0}*P_{1}*\ldots$. Then
\[
\overline{P}/\tor_{1}(\overline{P}) \cong \overline{P}.
\]
\end{lem}

\begin{proof}
Using proposition \ref{tor through products} and theorem \ref{fp tor len} we get that
\begin{align*}
\overline{P}/\tor_{1}(\overline{P}) & = (\overline{P}_{1}*\overline{P}_{2}*\ldots ) / \tor_{1}(\overline{P}_{1}*\overline{P}_{2}*\ldots)
\\ & \cong (\overline{P}_{1}/\tor_{1}(\overline{P_{1}}))*(\overline{P}_{2}/\tor_{1}(\overline{P_{2}}))* \ldots
\\ & \cong \{e\}*\overline{P}_{1}*\overline{P}_{2}*\ldots
\\ & \cong \overline{P}. \qedhere
\end{align*}
\end{proof}

\begin{cor}\label{inf tor len}
With $P$ as above, $\torlen(\overline{P})= \omega$.
\end{cor}

\begin{thm}\label{fp inf torlen}
There exists a $2$-generator recursive presentation $Q$ for which $\torlen(\overline{Q}) = \omega$.
\end{thm}

\begin{proof}
Take the countably generated recursive presentation $P$ from lemma \ref{non-hopf inf}, for which $\torlen(\overline{P})= \omega$ by corollary \ref{inf tor len}. But theorem \ref{3 gen torlen} gives that $\torlen(\overline{\tgen(P)})=\torlen(\overline{P})$ ($= \omega$). So taking $Q:=\tgen(P)$ gives a $2$-generator recursive presentation of a group $\overline{Q}$ with $\torlen(\overline{Q})=\omega$.
\end{proof}

\begin{quesb}
Does there exist a finitely presented group of infinite torsion length?
\end{quesb}

More ambitiously, is there a finitely presented (or even finitely generated) group $G$ with torsion for which $G/\tor_{1}(G) \cong G$? 

\subsection{Torsion length for solvable groups} \label{solsec}

\begin{prop}
Let $G$ be a group for which $\tor(G)$ forms a subgroup. Then $\torlen(G) \leq 1$.

\begin{proof}
Since $\tor(G)$ is a subgroup, it is immediate that $\tor(G)=\tor_{1}(G)$ (by lemma \ref{normal}). Suppose $x\tor_{1}(G)$ is a torsion element of $G/\tor_{1}(G)$. Then, there exists an $n$ such that $x^{n}\in \tor_{1}(G)$. However, since $\tor_{1}(G)=\tor(G)$, it follows that there exists an $m$ such that $x^{mn}=e$. Thus $x$ is torsion in $G$, and we have that $G/\tor_{1}(G)$ is torsion-free. Thus $\torlen(G)\leq 1$.
\end{proof}
\end{prop}

The set of torsion elements in a nilpotent group form a subgroup (\cite[5.2.7]{Rob}). Recall the family of presentations $P_{j,k,l}$ constructed in proposition \ref{first eg}.

\begin{prop}\label{solvable eg}
The group $\overline{P}_{j,k,l}$ is solvable if and only if $j=k=l=2.$ Moreover, $\overline{P}_{2,2,2}$ is polycyclic. 
\begin{proof}
Since $C_{j}*C_{k}$ embeds into $\overline{P}_{j,k,l}$, it follows from lemma \ref{freepfrees} that $\overline{P}_{j,k,l}$ is not solvable if either $j$ or $k$ is not $2$. Suppose $j=k=2$. Then, it is not hard to see that $\overline{P}_{2,2,l}$ surjects onto $C_{2}*C_{l}$ (to see this, introduce the relation $z^{l}=e$ in to the presentation $P_{2,2,l}$). Again, lemma \ref{freepfrees} tells us that $l$ must be $2$ if $\overline{P}_{2,2,l}$ were to be solvable.
\\We now show $\overline{P}_{2,2,2}$ is polycyclic. It follows from the presentation $P_{2,2,2}$ that the subgroup generated by $z^{2}$ is normal. Quotienting out by this subgroup, we get $C_{2}*C_{2}$, which is polycyclic by lemma \ref{c2sol}. Thus $\overline{P}_{2,2,2}$ is polycyclic. 
\end{proof}
\end{prop}

\begin{cor}\label{sol eg cor}
There exists a polycyclic group of torsion length $2$.
\end{cor}

While we suspect there exist solvable groups of arbitrary finite torsion length, we have been unable to construct them. 

\begin{quesb}
Is there a finitely generated solvable group of infinite torsion length?
\end{quesb}

\ 

\noindent \scriptsize{\textsc{Mathematics Department, University of Neuch\^{a}tel
\\Rue Emile-Argand 11, Neuch\^{a}tel, 2000, Switzerland.
\\maurice.chiodo@unine.ch
\vspace{5pt}
\\Department of Mathematics, Ben-Gurion University of the Negev,
\\P.O.B. 653, Beer-Sheva, 84105, Israel. 
\\vyas@math.bgu.ac.il}

\end{document}